\documentclass[12pt, twoside]{article}
\usepackage{amsmath}
\usepackage{amsthm}
\usepackage{amsfonts}
\usepackage{amssymb}
\usepackage{amsrefs}
\usepackage{epsfig}
\usepackage{graphicx}

\newcommand{\dimension}{n}

\newcommand{\ri}{\ensuremath{\rightarrow \infty}}

\newcommand{\dist}{\operatorname{dist}}

\newcommand{\half}{\frac{1}{2}}

\newcommand{\supp}{\operatorname{supp}}

\newcommand{\dx}{\;dx}

\renewcommand{\d}{\;d}

\newcommand{\ha}[1]{\frac{#1}{2}}
\newcommand{\ov}[1]{\frac{1}{#1}}

\newcommand{\expl}[2]{{\stackrel{\text{#1}}{#2}}}

\newcommand{\abs}[1]{\left|#1\right|}
\newcommand{\pth}[1]{\left(#1\right)}
\newcommand{\set}[1]{{\left\{#1\right\}}}

\newcommand{\ang}[1]{{\left\langle#1\right\rangle}}

\newcommand{\cl}[1]{\overline{#1}}

\newcommand{\no}[1]{\left\|#1\right\|}

\newcommand{\al}{\ensuremath{\alpha}}

\newcommand{\de}{\ensuremath{\delta}}
\newcommand{\vp}{\ensuremath{\varphi}}
\newcommand{\la}{\ensuremath{\lambda}}
\newcommand{\e}{\ensuremath{\varepsilon}}
\newcommand{\ve}{\ensuremath{\varepsilon}}

\newcommand{\Om}{\ensuremath{\Omega}}
\newcommand{\om}{\ensuremath{\omega}}

\newcommand{\R}{\ensuremath{\mathbb{R}}}

\newcommand{\Rd}{\ensuremath{{\mathbb{R}^{\dimension}}}}

\newcommand{\Rn}{\ensuremath{{\mathbb{R}^n}}}

\newcommand{\limto}[1]{\lim_{{#1} \rightarrow \infty}}

\newcommand{\bcs}{\begin{cases}}
\newcommand{\ecs}{\end{cases}}

\newcommand{\seq}[2]{\ensuremath{\{{#1}_{#2}\}_{{#2} = 1}^{\infty}}}

\newcommand{\pd}[2]{\frac{\partial {#1}}{\partial {#2}}}

\numberwithin{equation}{section}

\theoremstyle{definition}
\newtheorem{theorem}{Theorem}[section]

\newtheorem{lemma}[theorem]{Lemma}

\newtheorem{proposition}[theorem]{Proposition}
\newtheorem{definition}[theorem]{Definition}

\theoremstyle{definition}
\newtheorem{remark}[theorem]{Remark}
\newtheorem*{notation}{Notation}

\begin{document}

\baselineskip=17pt

\title{Long-time behavior of a Hele-Shaw type problem in random media}
\author{Norbert Po\v z\'ar\\
UCLA Department of Mathematics\\
Box 951555\\ Los Angeles, CA 90095-1555\\npozar@math.ucla.edu}
\maketitle

\abstract{We study the long-time behavior of an exterior Hele-Shaw problem in random media with a free boundary velocity that depends on position in dimensions $n \geq 2$. A natural rescaling of solutions that is compatible with the evolution of the free boundary leads to homogenization of the free boundary velocity. By studying a limit obstacle problem for a Hele-Shaw system with a point source, we are able to show uniform convergence of the rescaled solution to a self-similar limit profile and we deduce that the rescaled free boundary uniformly approaches a sphere.
}

\section{Introduction}

We consider a Hele-Shaw type problem (HS) in random media and study long-time behavior of the solutions. Let $n \geq 2$ and let $K \subset \Omega_0 \subset \Rd$, $K$ be a nonempty compact set and $\Om_0$ be a bounded open set such that $K$ and $\Omega_0$ have smooth boundaries (see Figure~\ref{fig:domain}).
\begin{figure}[t]
\centering
\includegraphics{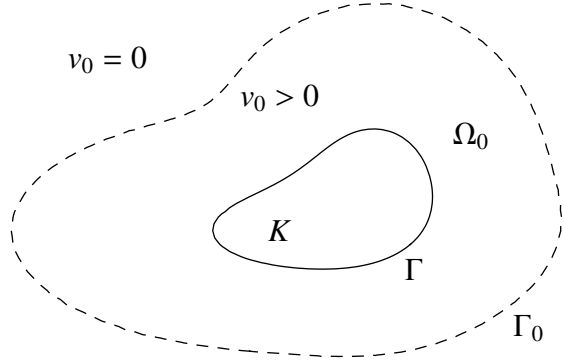}
\caption{Initial configuration}
\label{fig:domain}
\end{figure}
 Let $\Gamma = \partial K$ and $\Gamma_0 = \partial \Om_0$.
Formally, the Hele-Shaw-type problem is to find a function $v(x,t) : \Rd \times [0, \infty) \to \R$ that solves
\begin{equation}
\label{HSL}
\tag{HS}
\left\{
\begin{aligned} 
-\Delta v &= 0 && \text{in } \set{v > 0} \setminus K,\\
v &= 1 && \text{on } K,\\
v_t &= g(x, \om) |Dv|^2 && \text{on } \partial \set{v > 0},\\
v(\cdot, 0) &= v_0 && \text{on } \Omega_0 \setminus K, 
\end{aligned}
\right.
\end{equation}
where the initial data $v_0(x) : \Rd \to \R$ is the solution to
$$
\left\{
\begin{aligned} 
-\Delta v_0 &= 0 && \text{in } \Om_0 \setminus K,\\
v_0 &= 1 && \text{on } K,\\
v_0 &= 0 && \text{on } \partial \Omega_0.
\end{aligned}
\right.
$$
Here $D v$ and $\Delta v$ are, respectively, the gradient and Laplacian of $v$, with respect to the space variable $x$, and $v_t$ is the partial derivative of $v$ with respect to the time variable $t$. 

Let $(A, \mathcal F, \mu)$ be a probability space. We consider $g(x, \om) : \Rd \times A \to \R$ a continuous function in $x$ for a.e. $\om \in A$, satisfying  
\begin{equation}
\label{boundong}
0 < m \leq g(x,\om) \leq M \qquad \text{for all } x\in \Rd, \text{ a.e. } \om \in A,
\end{equation}
where $m$ and $M$ are positive constants. Furthermore, in order to observe some averaging behavior as $t \to \infty$, we assume that $g$ is \emph{stationary ergodic}. In other words, we assume that for every $x \in \Rd$ there exists an ergodic transformation $\tau_x: A \to A$ such that
\begin{align}
\label{statergodicg}
g(x + x', \om) = g(x', \tau_{x} \om) \qquad \text{for all $x,\, x' \in \Rd$ and a.e. $\om \in A$.}
\end{align}
We say that a measure preserving transformation $\tau_x$ is ergodic if the only invariant sets are either zero or full measure, that is, if $B \subset A$ such that $\tau_x(B) = B$ for all $x \in \Rd$, then $\mu(B) = 0$ or $1$.
For a more detailed discussion on stationary ergodic media, see for instance \cites{PV, CSW}.
 
The classical Hele-Shaw problem in two dimensions with $g \equiv 1$ was introduced in \cite{HS} modelling a slow movement of a viscous fluid injected in between two parallel plates close to each other that form the so called Hele-Shaw cell. This problem naturally generalizes to all dimensions $n \geq 1$.  

We study a Hele-Shaw type problem \eqref{HSL} describing a pressure-driven motion of a fluid in an inhomogeneous random medium where the velocity law of the fluid on the free boundary depends on the position. The homogenization of this problem was recently studied in \cite{KM}. Free boundary problems with similar  velocity laws have various applications in the plastics industry \cite{Richardson}, in electromechanical machining \cite{MG} and others. In fact, Hele-Shaw problem can be thought of as a quasi-stationary limit of the Stefan problem with a similar boundary velocity law, modelling heat transfer, see \cites{P, Rou}. 

In this paper, we are concerned with the behavior of the solution for large times. The main result, Theorem~\ref{uniformConvergenceOfGamma}, is the locally uniform convergence of a rescaled viscosity solution of the inhomogeneous problem \eqref{HSL}, 
\begin{align*}
v^\la(x,t) = \la^{(n-2)/n} v(\la^{1/n} x, \la t), \qquad \text{if $n \geq 3$},
\end{align*}
 (see section \ref{rescalling} for discussion of the rescaling and the case $n = 2$), which then satisfies 
\begin{align*}
v^\la_t = g(\la^{1/n} x, \om) \abs{Dv^\la}^2 \qquad \text{on $\partial \set{v^\la > 0}$}, 
\end{align*}
and its free boundary as $\la \to \infty$ to the self-similar solution of a Hele-Shaw problem with a point source, formally
\begin{equation}
\left\{
\begin{aligned} 
-\Delta v &= C \de && \text{in } \set{v > 0},\\
v_t &= \ov{\ang{\ov g}}|Dv|^2 && \text{on } \partial \set{v > 0},\\
v(\cdot, 0) &= 0 
\end{aligned}
\right.
\end{equation}
Here $\de$ is the Dirac $\de$-function with mass at the origin and the constant $C$ depends on $K$ and $n$. Constant $\ang{\ov{g}}$ represents the ``average'' of $\ov g$ and it will be properly defined later. 

The following heuristics can help with understanding the boundary velocity of the limit problem. For the moment, we suppose that $g$ is periodic and the solution $v$ is sufficiently smooth. Since the free boundary $\partial \set{v > 0}$ is a level set of $v$, its normal  velocity is given by $V = \frac{v_t}{\abs{Dv}}$. If we interpret \eqref{HSL} as a quasi-stationary limit of the one-phase Stefan problem where $v$ is the temperature and $-Dv$ is the heat flux, the quantity $\ov g$ represents the latent heat of phase transition depending on position (see \cite{P}). Indeed, for the free boundary to expand by $dx$, the heat flux must deliver $\ov{g} dx$ of heat into the interface. Since it delivers exactly $\abs{Dv} dt$ of heat in time $dt$, we have
\begin{align}
\label{heat}
V \sim \frac{dx}{dt} = \ov{\ov{g(x)}} \abs{Dv} = g(x) \abs{Dv}.
\end{align}
When we rescale the problem, the latent heat will simply average out and therefore
\begin{align*}
V = \ov{\ang{\ov{g}}} \abs{Dv}
\end{align*}
is the natural free boundary velocity of the limit solution.

 A similar result to ours was obtained in \cite{QV} for weak solutions of the homogeneous Hele-Shaw problem with free boundary velocity
\begin{align*}
V = \abs{D v}.
\end{align*}
  In the current situation, however, the velocity law of the free boundary depends on the position and therefore the techniques from \cite{QV} can provide only lower and upper bounds on the free boundary radius. This requires us to use a more refined method to prove the convergence of the solution to the self-similar asymptotic profile. We combine the strengths of two notions of solutions of \eqref{HSL} -- viscosity and weak -- using their correspondence studied in \cite{KM}. 

The weak solution $u(x,t) : \Rd \times [0, \infty) \to \R$ was introduced in \cite{EJ} by applying the Baiocchi transform 
\begin{align*}
u(x,t) = \int_0^t v(x,s) \d s.
\end{align*}
The function $u$ formally solves the Euler-Lagrange equation
\begin{align}
\label{eq:formalEL}
\left\{
\begin{aligned} 
-\Delta u &= -\ov{g(x, \om)} \chi_{\Rd \setminus \Omega_0} && \text{in } \set{u > 0},\\
u &= \abs{Du} = 0  && \text{on } \partial \set{u > 0},\\
u &= t && \text{on } K
\end{aligned}
\right.
\end{align}
of some obstacle problem. Since $g$ satisfies \eqref{boundong} and \eqref{statergodicg}, the subadditive ergodic theorem implies that there is a constant, denoted $\ang{\ov g}$, such that
\begin{align}
\label{convergenceofgtomean}
\int_\Rd \ov {g(\la^{1/n} x)} u(x) \dx \to \int_\Rn \ang{\ov g} u(x) \dx \qquad \text{for a.e. $\om \in A$}
\end{align}
as $\la \to \infty$.
This then leads to homogenization of the obstacle problem \eqref{eq:formalEL}.

While  in \cite{KM} the uniform convergence of the weak solutions in the homogenization limit follows more directly from the regularity of the obstacle problem \eqref{eq:formalEL}, the current situation is considerably different. As the fixed domain $K$ shrinks to the origin due to the rescaling, the rescaled weak solution gets singular. 
Thus the convergence of weak solutions as $\la \to \infty$ is restricted to a locally uniform convergence on compact sets not containing the origin and the obstacle problem regularity cannot be used directly.
We resolve this difficulty by introducing a limit obstacle problem and studying its wellposedness. The main challenge is posed by the singular behavior of the solution as $\abs{x} \to 0$. From this development, we can deduce the necessary local uniform convergence away from the origin.  As the last step, we extend the applicability of the tools developed in \cite{KM} to our situation with a singularity at the origin. This allows us to conclude that the rescaled free boundary converges uniformly to the free boundary of the asymptotic profile, a sphere. 

\vspace{0.1in}
The current method requires $g(x,\om)$ to be strictly positive, the main reason being the loss of uniqueness of viscosity solutions if $g$ becomes zero and the free boundary can stop. It is an interesting open problem if $g$ can reach zero yet $\ov g \in L^1_{loc}$ so that \eqref{convergenceofgtomean} still makes sense. It is currently under investigation by the author. The interpretation \eqref{heat} of \eqref{HSL} in terms of heat balance suggests that the homogenized velocity law $V = \ang{1/g}^{-1} \abs{Dv}$ continues to hold, at least for maximal viscosity solutions.

\vspace{0.1in}
A brief outline follows: Section 2 summarizes the two notions of solutions for problem \eqref{HSL} and their correspondence. In section 3, we introduce the rescaling of \eqref{HSL}, which is then followed by an investigation of radially symmetric test functions of section 4. The main contribution of the paper is then contained in sections 5 and 6, where we first define and study the limit obstacle problem and then show that the rescaled weak solutions of \eqref{HSL} converge to the unique solution of the limit obstacle problem. Finally, equipped with those results in section 7, we turn back to the viscosity notion to prove the uniform convergence of free boundaries to those of the limit problem.

\section{Hele-Shaw type problem}

\begin{notation}
 Since the arguments are independent of $\om$, we fix $\om \in A$ for which all \eqref{boundong}, \eqref{statergodicg} and \eqref{convergenceofgtomean} hold and we omit $\om$ in the rest of the paper. 

Throughout the article, we will make use of the standard notation for the bilinear form $a_\Omega(\cdot,\cdot)$ on $H^1(\Omega)$ and the scalar product $\ang{\cdot,\cdot}_\Omega$ on $L^2(\Omega)$ for some domain $\Omega$,
\begin{align*}
a_\Omega(u,v) = \int_\Omega Du \cdot Dv \dx, \qquad \ang{u, v}_\Omega = \int_\Omega uv \dx.
\end{align*}
Whenever $\Omega = \Rn$, we drop the subscript $\Omega$.

For a set $A$, we denote $A^c$ its complement. Given a nonnegative function $v$, we respectively define the positivity set of $v$ and the free boundary of $v$,
\begin{align*}
\Omega(v) := \set{(x,t) : v(x,t) > 0}, \qquad \Gamma(v) := \partial \Omega(v),
\end{align*}
and their time-slices
\begin{align*}
\Omega_t(v) := \set{x : v(x,t) > 0}, \qquad \Gamma_t(v) := \partial \Omega_t(v).
\end{align*}

$(f)_+$ will be the positive part of $f$, i.e. 
\begin{align*}
(f)_+ = \max \set{f, 0}.
\end{align*}
\end{notation}

Lastly, $B_r(x)$ is the open space ball of radius $r$ centered at $x$, 
\begin{align*}
B_r(x) = \set{y \in \Rn: \abs{y - x} < r}.
\end{align*}

This section reviews the two notion of solutions of \eqref{HSL} that will be used in this paper.

The notion of \emph{viscosity solutions} for \eqref{HSL} was introduced in \cite{Kim07} and used in \cite{KM} to prove uniform convergence of the free boundary in the homogenization limit. We will consider solutions on the spacetime cylinder $Q = (\Rd \setminus K) \times [0, \infty)$.

\begin{definition}
\label{def:viscositySubsolution}
A nonnegative upper-semicontinuous function $v$ defined in $Q$ is a \emph{viscosity subsolution} of \eqref{HSL} if
\begin{enumerate}
\item for each $T > 0$, the set $\cl{\Omega(v)} \cap \set{t \leq T} \cap Q$ is bounded, and
\item for every $\phi \in C^{2,1}_{x,t}(Q)$ such that $v - \phi$ has a local maximum in $\cl{\Omega(v)} \cap \set{t \leq t_0} \cap Q$ at $(x_0, t_0)$, the following two conditions hold:
\begin{enumerate}
\item if $v(x_0,t_0) > 0$ then $- \Delta \phi(x_0, t_0) \leq 0$,
\item if $(x_0, t_0) \in \Gamma(v)$, $\abs{D\phi}(x_0, t_0) \neq 0$ and $- \Delta \phi (x_0,t_0) > 0$, then 
\begin{align*}
(\phi_t - g(x_0) \abs{D\phi}^2)(x_0,t_0) \leq 0.
\end{align*}
\end{enumerate}
\end{enumerate}
\end{definition}

\begin{definition}
A nonnegative lower-semicontinuous function $v$ defined in $Q$ is a \emph{viscosity supersolution} of \eqref{HSL} if
for every $\phi \in C^{2,1}_{x,t}(Q)$ such that $v - \phi$ has a local minimum in $\cl{\Omega(v)} \cap \set{t \leq t_0} \cap Q$ at $(x_0, t_0)$, the following two conditions hold:
\begin{enumerate}
\item if $v(x_0,t_0) > 0$ then $- \Delta \phi(x_0, t_0) \geq 0$,
\item if $(x_0, t_0) \in \Gamma(v)$, $\abs{D\phi}(x_0, t_0) \neq 0$ and $- \Delta \phi (x_0,t_0) < 0$, then 
\begin{align*}
(\phi_t - g(x_0) \abs{D\phi}^2)(x_0,t_0) \geq 0.
\end{align*}
\end{enumerate}
\end{definition}

Now we can define viscosity sub- and supersolutions with appropriate initial and boundary data.
\begin{definition}
A viscosity subsolution $v$ of \eqref{HSL} in $Q$ is a \emph{viscosity subsolution of \eqref{HSL} in $Q$ with initial data $v_0$ and boundary data $1$} if 
\begin{enumerate}
\item  $v$ is upper-semicontinuous in $\cl{Q}$, $v = v_0$ at $t = 0$ and $v \leq 1$ on $\Gamma$,
\item $\cl{\Omega(v)} \cap \set{t = 0} = \cl{\set{x : v_0(x) > 0}}$.
\end{enumerate} 
\end{definition}

\begin{definition}
A viscosity supersolution $v$ of \eqref{HSL} in $Q$ is a \emph{viscosity supersolution of \eqref{HSL} in $Q$ with initial data $v_0$ and boundary data $1$} if 
\begin{enumerate}
\item  $v$ is lower-semicontinuous in $\cl{Q}$, $v = v_0$ at $t = 0$ and $v \geq 1$ on $\Gamma$,
\end{enumerate}
\end{definition}

Finally, we can define a viscosity solution:
\begin{definition}
A viscosity supersolution $v$ of \eqref{HSL} in $Q$ (with initial data $v_0$ and boundary data $1$) is a \emph{viscosity solution of \eqref{HSL} in $Q$ (with initial data $v_0$ and boundary data $1$)} if 
\begin{align*}
v^\star(x,t) := \limsup_{(y,s)\to (x,t)} v(y,s)
\end{align*}
is a viscosity subsolution of \eqref{HSL} in $Q$ (with initial data $v_0$ and boundary data $1$).
\end{definition}

The following comparison principle is due to \cite{Kim07}:
\begin{theorem}[\cite{Kim07}*{Theorem 1.7}]
Let $v_1$ and $v_2$ be respectively viscosity sub- and supersolution of \eqref{HSL}. If $\cl {\set{v_1(\cdot, 0) > 0}}$ is bounded, $v_1(x, 0) < v_2(x,0)$ in $\set{v_1(\cdot, 0) > 0}$ and $v_1 < v_2$ on $\partial K \times [0, \infty)$, then 
\begin{align*}
v_1 \leq v_2  \qquad \text{on $K^c \times [0, \infty)$}.
\end{align*}
\end{theorem}

The second notion of solutions is the one introduced by \cite{EJ}, given by an obstacle problem for each time for a new variable $u(x,t) = \int_0^t v(x,s) \d s$:

\begin{definition}
The function $u(x,t)$ is called the \emph{weak solution} of problem \eqref{HSL} if for every $t\geq 0$, $w = u(\cdot, t)$ solves the obstacle problem
\begin{equation}
\label{OP}
\bcs w \in \mathcal{K}_t,\\ a(w, \vp - w) \geq \ang{ -\ov{g(x)} \chi_{\Rd\setminus\Om_0}, \vp - w } & \text{for all } \vp \in \mathcal{K}_t, \ecs
\end{equation}
where 
$$
\mathcal{K}_t = \set{\vp \in H_0^1(\Rd) ,\ \vp \geq 0 \text{ in } \Om,\ \vp = t \text{ on } K}.
$$
\end{definition}

The theory of the obstacle problem \eqref{OP} is well understood. We refer the reader to \cite{Rodrigues} for instance. We have the following:
\begin{proposition}
Suppose that $K \subset \Omega_0 \subset \Rn$, $\Omega_0$ is bounded open set and $K$ is a compact set with $\partial K \in C^{1,1}$. Then \eqref{OP} has a unique solution $u(\cdot, t) \in H^1_0(\Rd)$ for all $t > 0$. For every $T > 0$ there exists $R(T)$ such that 
\begin{align*}
\Omega_t(u) \subset B_R(0) \qquad \text{for all $t \in [0,T]$}.
\end{align*} 
Moreover, for any $1 < p < \infty$, there is a constant $C = C(m, n, p)$ such that 
\begin{align*}
\no{u(\cdot,t)}_{W^{2,p} (\Rd \setminus K)} \leq Ct \qquad \text{for all $t \geq 0$}.
\end{align*}
Finally, the function $t \mapsto u(x,t)$ is Lipschitz continuous with
\begin{align*}
\abs{u(x, t_1) - u(x, t_2)} \leq \abs{t_1 - t_2} \qquad \text{for all $0 < t_1 \leq t_2 $},
\end{align*}
for all $x \in \Rd$. 
\end{proposition}

One of the main tools used in this paper is the following correspondence of weak and viscosity solutions proved in \cite{KM}:
\begin{theorem}[cf. \cite{KM}*{Theorem 3.1}]
\label{th:equivalency}
Let $u(x,t)$ be the unique solution of \eqref{OP}. For each $t > 0$, let $v(\cdot,t)$ be the solution of
$$
\bcs
\Delta v(\cdot, t) = 0 & \text{in } \Om_t(u) \setminus K,\\
v = 1 & \text{on } K\\
v = 0 & \Gamma_t(u),
\ecs
$$
i.e. for every $t > 0$ define $v(\cdot, t)$ as the supremum of all lower semicontinuous functions $w(x)$ for which there exists some $s$,  $0 < s < t$, such that
$$
\bcs
- \Delta w \leq 0 &\text{ in } \Omega_s(u),\\
 w = 1 &\text{ on } K,\\
\supp w \subset \Omega_s(u).
\ecs
$$
Then $v$ is a viscosity solution of \eqref{HSL}. Moreover,
\begin{align*}
v(x,t) = \partial_t^- u(x,t),
\end{align*}
the left time derivative of $u$.
\end{theorem}

\section{Rescaling}
\label{rescalling}

We are interested in the behavior of \eqref{HSL} as $t \ri$. The goal is to prove that a rescaled solution converges to the self-similar solution of the homogeneous Hele-Shaw problem with a point delta source. As was observed in \cite{QV}, the behavior depends on the dimension and the necessary rescaling is more involved in the two-dimensional case.

\subsection{Case $n \geq 3$}
The heat interpretation \eqref{heat} of the free boundary velocity from the introduction suggests that the radius of the free boundary should behave as $\sim t^{1/n}$ and thus we use the rescaling
\begin{align}
\label{eq:rescaling}
v^\la(x,t) = \la^{(n-2)/n} v(\la^{1/n} x, \la t).
\end{align}
If we define $K^\la := \frac{K}{\la^{1/n}}$ and $\Om^\la_0 := \frac{\Om_0}{\la^{1/n}}$,  we derive the following problem for $v^\la$ from \eqref{HSL}:
\begin{equation}
\label{THSL}
\bcs
-\Delta v^\la = 0 & \text{in } \Omega(v^\la) \setminus K^\la,\\
v^\la = \la^{(n-2)/n} & \text{on } K^\la,\\
v^\la_t = g^\la(x) \abs{Dv^\la}^2 & \text{on } \Gamma(v^\la),\\
v^\la(\cdot, 0) = v^\la_0 & \text{on } \Omega_0^\la \setminus K^\la, 
\ecs
\end{equation}
where $g^\la(x) = g(\la^{1/n} x)$ as can be easily seen from
\begin{align}
\label{eq:scalingOfg3}
\begin{aligned}
\la^{(n-2)/n + 1} v_t(\la^{1/n} x, \la t) &= g^\la(x) \la^{2 (n-2)/n + 2/n}\abs{(Dv)(\la^{1/n} x, \la t)}\\
v_t(\la^{1/n} x, \la t) &= g^\la(x) \abs{(Dv)(\la^{1/n} x, \la t)}
\end{aligned}
\end{align}
compared to the boundary condition in \eqref{HSL}
\begin{align*}
v_t(\la^{1/n} x, \la t) &= g(\la^{1/n} x) \abs{(Dv)(\la^{1/n} x, \la t)}.
\end{align*}

The rescaled problem \eqref{THSL} implies the weak formulation for the rescaled $u^\la(\cdot, t) = \int_0^t v^\la(\cdot, s) \d s$, the obstacle problem 
\begin{equation}
\label{TOP}
\bcs w \in \mathcal{K}_t^\la\\ 
a\pth{w, \vp - w} \geq \ang{ -\ov{g^\la} \chi_{\Rn \setminus\Om^\la_0}, \vp - w} & \text{for all } \vp \in \mathcal{K}_t^\la, \ecs
\end{equation}
for each $t > 0$, 
where 
$$
\mathcal{K}_t^\la = \set{\vp \in H^1(\Rn) ,\ \vp \geq 0,\ \vp = \la^{(n-2)/n} t \text{ on } K^\la}.
$$
The relation of $u^\la$ to $u$ can be deduced from the definition of $u^\la$ as
\begin{equation}
\label{rescaledU}
\begin{split}
u^\la(x,t) &= \int_0^t v^\la(x,s) \d s \\
&= \la^{(n-2)/n} \int_0^{t} v(\la^{1/n} x, \la s) \d s \\
&= \la^{(n-2)/n} \int_0^{\la t} v(\la^{1/n} x, \sigma) \frac{\d \sigma}{\la} \\
&= \la^{-2/n} u(\la^{1/n} x, \la t).
\end{split}
\end{equation}

The following observation will be useful:
\begin{lemma}
\label{rescaledBound}
The harmonic function $p(x) = C|x|^{2-n}$ is invariant under the rescaling \eqref{eq:rescaling}. In particular,
if $v(x,t) \leq C|x|^{2-n}$ for all $t > 0$ then also
$$
v^\la(x,t) \leq C|x|^{2-n} \quad \text{for all } t > 0.
$$
\end{lemma}

\begin{proof}
Observe that
$$
p^\la(x) = C \la^{(n-2)/n} |\la^{1/n} x|^{2-n} = C|x|^{2-n}.
$$
\end{proof}

\subsection{Case $n = 2$}

Since the rescaling \eqref{eq:rescaling} does not recover the logarithmic singularity at the origin in dimension $n = 2$, we use a different rescaling that preserves the system:
\begin{align*}
v^\la(x,t) = \log \mathcal{R}(\la) v ( \mathcal{R}(\la) x, \la t),
\end{align*}
where $\mathcal{R}(\la)$ is the unique solution of 
\begin{align}
\label{eq:r2}
\mathcal{R}^2 \log \mathcal{R} = \la
\end{align}
for any $\la > 0$. In fact, the solution can be found explicitly in terms of the Lambert W function, $\mathcal{R}(\la) = \exp\pth{\half W(2 \la)}$. It is easy to see that $\lim_{\la\to \infty} \mathcal{R}(\la) = \infty$. Taking a logarithm of \eqref{eq:r2} gives
\begin{align*}
2 \log \mathcal{R}(\la) \pth{1 + \frac{\log \log \mathcal{R}(\la)}{2 \log \mathcal{R}(\la)} - \frac{\log \la}{2 \log \mathcal{R}(\la)}} = 0.
\end{align*}
We deduce 
\begin{align}
\label{eq:rescaling2LogAsymptotics}
\lim_{\la\to \infty} \frac{\log \la}{2 \log \mathcal{R}(\la)} = 1,
\end{align} 
which together with \eqref{eq:r2} yields
\begin{align}
\label{eq:rescaling2Asymptotics}
\lim_{\la \to \infty} \frac{\mathcal{R}(\la)}{\mathcal{R}_\infty(\la)} = 1, \qquad \mathcal{R}_\infty(\la) = \pth{\frac{2\la}{\log \la}}^{1/2}.
\end{align}

In this case, we define $K^\la := \frac{K}{\mathcal{R}(\la)}$ and $\Om^\la_0 := \frac{\Om_0}{\mathcal{R}(\la)}$. Since $v$ satisfied \eqref{HSL},  $v^\la$ satisfies
\begin{equation}
\label{THSL2}
\bcs
-\Delta v^\la = 0 & \text{in } \Omega(v^\la) \setminus K^\la,\\
v^\la = \log \mathcal{R}(\la) & \text{on } K^\la,\\
v^\la_t = g^\la(x) \abs{Dv^\la}^2 & \text{on } \Gamma(v^\la),\\
v^\la(\cdot, 0) = v^\la_0 & \text{on } \Omega_0^\la \setminus K^\la. 
\ecs
\end{equation}
Here $g^\la(x) = g(\mathcal{R}(\la) x)$, which can be derived as in \eqref{eq:scalingOfg3} together with \eqref{eq:r2}.

The rescaled problem \eqref{THSL2} implies a weak formulation for the rescaled $u^\la(x,t) = \int_0^t v^\la(x, s) \d s$ of the form \eqref{TOP}
where now
$$
\mathcal{K}_t^\la = \set{\vp \in H^1(\Rn) ,\ \vp \geq 0,\ \vp = \log \mathcal{R}(\la) t \text{ on } K^\la}.
$$
This induces scaling for weak solution $u$ analogous to \eqref{rescaledU},
\begin{align*}
u^\la(x,t) &= \frac{\log \mathcal{R}(\la)}{\la} u(\mathcal{R}(\la) x, \la t).
\end{align*}

\begin{remark}
Note that the scaling in dimensions $n \geq 3$ and the scaling in dimension $n = 2$ are qualitatively different. Indeed, in dimension $n \geq 3$ a solution with a point source and a solution with a finite source have the same asymptotic speed of the free boundary $\approx t^{1/n}$.
That is not true in dimension $n = 2$ where a point source solution has asymptotic speed $\approx t^{1/2}$, but a solution with a finite source has a slower asymptotic speed $\approx \pth{\frac{t}{\log t}}^{1/2}$.
\end{remark}

\section{Comparison with radially symmetric solutions}

Radially symmetric solutions for the Hele-Shaw problem, derived in \cite{QV}, will serve as test functions in our arguments. A radially symmetric solution in the domain $|x| \geq  a$, $t \geq 0$ is a pair of functions $p(x,t)$ and $R(t)$, where $p$ is of the form
\begin{subequations}
\label{eq:radialSolution}
\begin{equation}
\label{radialSolution}
p(x,t) = \frac{A a^{2-n} \pth{|x|^{2-n} - R^{2-n}(t)}_+}{a^{2-n} - R^{2-n}(t)} \qquad \text{if $n \geq 3$},
\end{equation}
or
\begin{equation}
\label{radialSolution2}
p(x,t) = \frac{A \pth{\log \frac{R(t)}{\abs{x}}}_+}{\log \frac{R(t)}{a}} \qquad \text{if $n = 2$},
\end{equation}
\end{subequations}
and $R(t)$ satisfies a certain algebraic equation (see \cite{QV} for details, we will be interested only in the behavior as $t \to \infty$).
This solution satisfies the boundary conditions
\begin{align}
\label{eq:radialBoundaryCondition}
\begin{aligned}
p(x, t) &= A a^{2-n} & \text{for } |x| &= a > 0,\\
p(x,t) &= 0 &\quad \text{for } |x| &= R(t),\\
R'(t) &= \ov L |D p|, &\quad \text{for } |x| &= R(t),\\
R(0) &= b > a.
\end{aligned}
\end{align}
Also
\begin{subequations}
\label{eq:radiusBehavior}
\begin{align}
\label{eq:radiusBehavior3}
\limto{t} \frac{R(t) }{c_\infty t^{1/n}} = 1, \qquad c_\infty = \pth{\frac{A n (n-2)}{L}}^{1/n} \qquad \text{if } n\geq 3, 
\end{align}
or
\begin{align}
\label{eq:radiusBehavior2}
\limto{t} \frac{R(t) }{c_\infty \pth{\frac{t}{\log t}}^{1/2}} = 1, \qquad c_\infty = 2 \sqrt{A/L} \qquad \text{if } n = 2. 
\end{align}
\end{subequations}
In dimension $n = 2$, we will also make use of the limit behavior
\begin{align}
\label{eq:radiusLogBehavior2}
\lim_{t \to \infty}  \frac{\log R(t)}{\log t} = \half.
\end{align}
\begin{lemma}
\label{subSuperSolutions}
For $L = \ov m$ (resp. $L = \ov M$), with $m$, $M$ defined in \eqref{boundong}, $p(x,t)$ is a viscosity subsolution (resp. supersolution) of \eqref{HSL} in $Q$. 
\end{lemma}
\begin{proof}
Assume that $\ov L = m$. The case $\ov L = M$ is similar. We will show that $p(x,t)$ satisfied the conditions in definition \ref{def:viscositySubsolution}. 

$p(x,t)$ is clearly continuous and (a) immediately follows from \eqref{radialSolution} and \eqref{eq:radiusBehavior}.

As for (b), (i) follows from the fact that $p(\cdot, t)$ is harmonic in $\Omega_t(p)$ for every $t \geq 0$.

To show (ii), let $\phi$ be a smooth function such that $p - \phi$ has a local maximum $0$ in $\cl{\Om(p)} \cap \set{t \leq t_0}$ at $(x_0,t_0) \in \Gamma(p)$ such that $|D\phi|(x_0,t_0) \neq 0$ and $- \Delta \phi (x_0,t_0) > 0$. Let $B_r \subset \Om_{t_0}(p)$ be a ball touching $\Gamma_{t_0}(p)$ from inside at $(x_0,t_0)$ in which $-\Delta \phi > 0$ (by smoothness of boudaries and continuity of $\Delta \phi$ there is such a ball). Then by the Hopf's Lemma (see \cite{Evans}, ch. 6.4.2), since $-\Delta (\phi - p)  = - \Delta \phi < 0$, 
$$
\pd{( \phi - p)}{\nu} > 0 \qquad \text{at } (x_0,t_0),
$$
where $\nu$ is the inner normal. Thus together with the fact that the boundaries are smooth level sets, we conclude that
$$
|Dp|(x_0,t_0) = \pd{p}{\nu} < \pd{\phi}{\nu} = |D\phi|(x_0,t_0).
$$
 The normal velocity of the free boundary is given as $V_n = p_t / |Dp|$ (level sets), we have
$$
\frac{p_t}{|Dp|} (x_0,t_0) \geq \frac{\phi_t}{|D\phi|}(x_0,t_0).
$$
Hence finally
\begin{align*}
(\phi_t - g (x_0) |D\phi|^2) (x_0,t_0) &< |D\phi| \pth{\frac{p_t}{|Dp|} - g(x_0) |Dp| } (x_0,t_0)\\
&\expl{\eqref{boundong}}{\leq}|D\phi| \pth{\frac{p_t}{|Dp|} - m |Dp| }(x_0,t_0) \expl{\eqref{eq:radialBoundaryCondition}}{=} 0.
\end{align*}
\end{proof}

Comparison with the radially symmetric sub- and supersolution provides us with the following straightforward result:
\begin{lemma}
\label{th:boundaryBound}
Let $v$ be a viscosity solution of \eqref{HSL}. There exists $t_0 > 0$ and constants $\rho_1$ and $\rho_2$, $0 < \rho_1 < \rho_2$, such that 
$$
\rho_1 t^{1/n} < \min_{\Gamma_t(v)} |x| \leq \max_{\Gamma_t(v)} |x| < \rho_2 t^{1/n} \qquad \text{if } n \geq 3,
$$
or 
$$
\rho_1 \mathcal{R}(t) < \min_{\Gamma_t(v)} |x| \leq \max_{\Gamma_t(v)} |x| < \rho_2 \mathcal{R}(t) \qquad \text{if } n = 2,
$$

for $t \geq t_0$ and
\begin{align*}
\max_{\Gamma_t(v)} \abs{x} < \rho_2
\end{align*}
for $0 \leq t \leq t_0$.

Moreover there is a constant $C > 0$ such that
\begin{align*}
0 \leq v(x,t) \leq C |x|^{2-n}.
\end{align*}
\end{lemma}
\begin{proof}
\begin{figure}[t]
\centering
\includegraphics{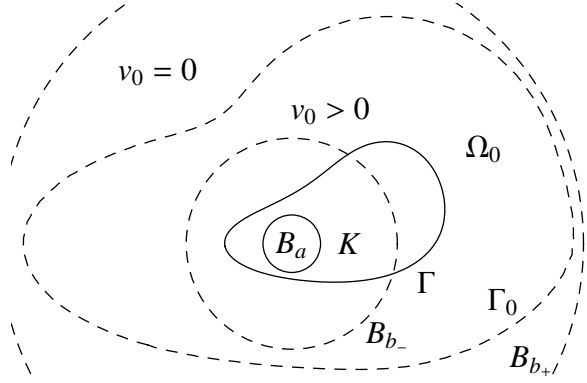}
\caption{Arrangement of domains for comparison}
\label{fig:domainBalls}
\end{figure}Given problem \eqref{HSL}, we can find constants $a$, $b_\pm$ such that $B_a(0) \subset K$ and $\cl{B_{b_-}(0)} \subset \Om_0 \subset \cl{\Om_0} \subset B_{b_+}(0)$ (see Fig.~\ref{fig:domainBalls}). Set also $L_- = \ov m$, $L_+ = \ov M$.  Using the maximum principle for harmonic functions, we can also find constants $A_-$ and $A_+$ so that $p_-$ and $p_+$, the radially symmetric subsolution and supersolution, respectively, given in \eqref{eq:radialSolution} that satisfy the boundary conditions \eqref{eq:radialBoundaryCondition} with the respective constants, also satisfy
$$
p_- \prec v \prec p_+ \qquad \text{on } \pth{K^c \times \set{0}} \cup \pth{\Gamma \times (0,\infty)}.
$$
Then by the comparison principle for viscosity solutions, Theorem 2.6 in \cite{KM}, we have 
$$
p_- \prec v \prec p_+ \qquad \text{on } K^c \times (0,\infty).
$$
and also
\begin{align*}
v(x,t) \leq p_+(x,t) < C |x|^{2 - n},
\end{align*}
for $C > 0$ large enough and
$$
R_-(t)  < \min_{\Gamma_t(v)} |x| \leq \max_{\Gamma_t(v)} |x| < R_+(t),
$$
where $R_\pm$ are the radii of free boundaries of functions $p_\pm$, respectively, and their limit behavior is given by \eqref{eq:radiusBehavior}.
\end{proof}

Lemma~\ref{th:boundaryBound} is all we need to prove the crucial theorem from \cite{QV}:
\begin{theorem}[Near-field limit]
\label{nearFieldLimit}
The viscosity solution $v(x,t)$ of the Hele-Shaw-like problem \eqref{HSL} converges to the unique solution $P(x)$ of the exterior Dirichlet problem
\begin{equation}
\label{exteriorDirichletProb}
\bcs
\Delta P = 0, & x \in \R^n \setminus K,\\
P = 1, & x \in \Gamma,\\
\limto{|x|} P(x) = 0, & \text{if } n \geq 3, \text{ or}\\
P \text{ is bounded} & \text{if } n = 2.
\ecs
\end{equation} 
as $t \to \infty$ uniformly on compact subsets of $\cl{K^c}$.
\end{theorem}

\begin{proof}
See proof of \cite{QV}*{Theorem 4.1}.
\end{proof}

The following constant $C_*$ characterizes the singularity of the limit solution:

\begin{lemma}
\label{constantSingularity}
There exists a constant $C_* = C_*(K)$ such that the solution $P$ of problem \eqref{exteriorDirichletProb} satisfies
$$
\limto{|x|} |x|^{n-2} P(x) = C_*.
$$
\end{lemma}

\begin{proof}
See Lemma 4.5 in \cite{QV}.
\end{proof}

Finally, we also need to improve on the convergence result of radially symmetric solutions from \cite{QV}:
\begin{lemma}
\label{th:radialConvergence}
Let $p(x,t)$ be a radially symmetric solution \eqref{eq:radialSolution} of the Hele-Shaw problem satisfying the boundary conditions \eqref{eq:radialBoundaryCondition} with constants $A,\ a,\ b$ and $L$. Then the rescaled solutions $p^\la(x,t)$ converge locally uniformly on the set $(\Rd \setminus \set{ 0}) \times [0, \infty)$ to the solution of the Hele-Shaw problem with a point source,
\begin{align*}
V(x,t) &= V_{A,L}(x,t) =
\bcs 
A \pth{\abs{x}^{2-n} - \rho^{2-n}(t)}_+,& n \geq 3\\
 A\pth{\log \frac{\rho(t)}{\abs{x}}}_+, & n = 2,
\ecs
\end{align*}
where
\begin{align*}
\rho(t) = \rho_L(t) =
\bcs
 \pth{\frac{A n (n-2) t}{L}}^{1/n}, & n \geq 3\\
 \pth{\frac{2 A t}{L}}^{1/2},& n = 2
\ecs
\end{align*}
\end{lemma}

\begin{proof}
We shall show uniform convergence on sets $\set{(x,t): \abs{x} \geq \ve,\ 0 \leq t \leq T}$ for some $\ve, \ T > 0$. Let $t_0 = \rho^{-1}(\ve/2)$.  Starting with $n \geq 3$, we have

\begin{align*}
p^\la(x,t) = \frac{A a^{2-n}\pth{\abs{x}^{2-n} - \pth{\frac{R(\la t)}{(\la t)^{1/n}} t^{1/n}}^{2-n}}_+}{a^{2-n} - R^{2-n}(\la t)}.
\end{align*}

There are two cases:
\begin{enumerate}
\item $t_0 \leq t \leq T$

Due to \eqref{eq:radiusBehavior}, we have
\begin{align*}
\frac{R(\la t)}{\la^{1/n}} = \frac{R(\la t)}{(\la t) ^{1/n}} t^{1/n} &\to \pth{\frac{A a^{n-2} n (n-2) t}{L}}^{1/n} = \rho(t),\\
R^{2-n}(\la t) &\to 0,
\end{align*}
as $\la \to \infty$ uniformly on $t_0 \leq t \leq T$. That shows uniform convergence $p^\la \to V$ on $\set{\abs{x} \geq \ve, \ t_0 \leq t \leq T}$.
\item $0 \leq t \leq t_0$

Clearly $V(x,t) = 0$ in $\set{\abs{x} \geq \ve,\ 0 \leq t \leq t_0}$. Since
\begin{align*}
\frac{R(\la t)}{\la^{1/n}} \leq \frac{R(\la t_0)}{\la^{1/n}} < \rho(t_0) + \frac{\ve}{2} = \ve
\end{align*} 
for all $\la$ large enough, we see that $p^\la = 0 = V$ on $\set{\abs{x} \geq \ve, \ 0 \leq t \leq t_0}$ for all $\la$ large. 
\end{enumerate}

In dimension $n=2$, the rescaling yields
\begin{align*}
p^\la(x,t) = A \frac{\pth{\log \frac{R(\la t)}{\mathcal{R}(\la) \abs{x}}}_+}{\frac{\log \frac{R(\la t)}{a}}{\log \mathcal{R}(\la)}}.
\end{align*}
Again, we split the proof in two cases:
\begin{enumerate}
\item $ t_0 \leq t \leq T$

Rewriting $R(\la t)/\mathcal{R}(\la)$ and using \eqref{eq:radiusBehavior2} and \eqref{eq:rescaling2Asymptotics}, we obtain
\begin{align*}
\frac{R(\la t)}{\mathcal{R}(\la)} &= \frac{R(\la t)}{\pth{\frac{\la t}{\log \la t}}^{1/2}} \cdot \frac{\pth{\frac{\la t}{\log \la t}}^{1/2}}{\pth{\frac{\la}{\log \la}}^{1/2}} \cdot \frac{\pth{\frac{\la}{\log \la}}^{1/2}}{\mathcal{R}(\la)} \to 2\pth{\frac{A}{L}}^{1/2} \cdot t^{1/2} \cdot \frac{1}{\sqrt{2}} \\
&= \pth{\frac{2 A t}{L}}^{1/2} 
\end{align*} 
as $\la \to \infty$ uniformly for $t_0 \leq t \leq T$.

Similarly, \eqref{eq:radiusLogBehavior2} and \eqref{eq:rescaling2LogAsymptotics} lead to
\begin{align*}
\frac{\log \frac{R(\la t)}{a}}{\log \mathcal{R}(\la)} = \frac{\log R(\la t) - \log a}{\log \la t} \cdot \frac{\log \la t}{\log \la} \cdot \frac{\log \la}{\log \mathcal{R}(\la)} \to \half \cdot 1 \cdot 2 = 1
\end{align*}
as $\la \to \infty$ uniformly for $t_0 \leq t \leq T$.
This proves uniform convergence of $p^\la \to V$ on $\set{\abs{x} \geq \ve, \ t_0 \leq t \leq T}$.
\item $0 \leq t \leq t_0$

Argue as in the case $n \geq 3$, (b).
\end{enumerate}
\end{proof}

\section{The limit problem}
Our current task is a characterization of the limit of rescaled weak solutions $u^\la$ as $\la \to \infty$. We want to show that the limit satisfies a certain obstacle problem that can be interpreted as a Hele-Shaw problem with a point source. Existence and uniqueness of such a problem in 2 dimensions was studied in \cite{CJ}. Our situation requires extending the definition to all dimensions $n \geq 3$.

First define $U_{A,L}(x,t)$ to be the Baiocchi transform of $V_{A,L}(x,t)$, introduced in Lemma~\ref{th:radialConvergence} (see proof of Theorem~\ref{th:LPwellposed} for derivation):
\begin{align}
\label{eq:U}
&U_{A, L}(x,t) =\nonumber\\
&=\bcs
\pth{A t |x|^{2-n} + \frac{L}{2n}|x|^2 - \half \pth{A n t}^{2/n} \pth{\frac{n-2}{L}}^{(2-n)/n}}_+ & \text{if } n \geq 3,\\
\pth{\frac{A}{2}t \log \frac{2At}{L e \abs{x}^2} + \frac{L\abs{x}^2}{4}}_+ & \text{if } n = 2.
\ecs
\end{align}

We say that $U(x,t)$ is a solution of the \emph{limit problem} if for every $t \in [0,T]$, $U(\cdot, t)$ satisfies the following obstacle problem:
\begin{equation}
\label{limitProblem}
\left\{
\begin{aligned}
w &\in \mathcal{K}_t,\\
a(w, \phi) &\geq \ang{-L, \phi}, &\forall &\phi \in V,\\
a(w, \psi w) &= \ang{-L, \psi w}, &\forall &\psi \in W,
\end{aligned}
\right.
\end{equation}
where 
\begin{align*}
\mathcal{K}_t = \Biggl\{\vp &\in \bigcap_{\substack{\e > 0}} H^1(\Rn \setminus B_\e) \cap C(\Rn \setminus B_\e):\\
&\vp \geq 0, \ \lim_{|x| \to 0} \frac{\vp(x)}{U_{A,L}(x,t)} = 1\Biggr\},
\end{align*}
\begin{subequations}
\label{eq:limitProblemVW}
\begin{align}
V = \set{\phi \in H^1(\Rn),\ \phi \geq 0,\ \phi = 0 \text{ on } B_\ve(0) \text{ for some } \ve > 0}.
\end{align}
and
\begin{align}
W = V \cap C^1(\Rn).
\end{align}
\end{subequations}

\begin{theorem}
\label{th:LPwellposed}
For given $A > 0$ and $L > 0$, $U_{A,L}(x,t)$ defined in \eqref{eq:U} is the unique solution of the problem \eqref{limitProblem} 

\end{theorem}

\begin{proof}
We first verify that $U_{A,L}$ is a solution. It can be found by integrating the function $V_{A,L}(x,t)$ from Lemma~\ref{th:radialConvergence}, 
\begin{align*}
U(x,t) &= \int_0^t V(x,s) \d s = \bcs
0 & t \leq s(x),\\
\int_{s(x)}^t V(x,t) \d t, & t > s(x) 
\ecs,
\end{align*}
where 
$$
s(x) = \rho^{-1}(\abs{x}).
$$
It is clear that $\Delta U = L$ in $U > 0$ and $U = \pd{U}{\nu} = 0$ on $\Gamma_t(U)$, i.e. when $|x| = \rho(t)$. A straightforward application of the Green's theorem then yields that $U$ satisfies the variational equality and inequality in \eqref{limitProblem}.

To prove uniqueness, we will use comparison for an obstacle problem and for this we need the following observation:
\begin{lemma}
\label{limitProblemSatisfiesObstacleProblem}
If $w$ satisfies the obstacle problem \eqref{limitProblem}, it also for every $a > 0$ satisfies the obstacle problem 
$$
a_\Om(w, \vp - w) \geq \ang{-L, \vp - w}_\Om \qquad \text{for all } \vp \in \mathcal{K},
$$
on $\Om = \Rn \setminus B_a$ , where
$$
\mathcal{K} = \set{\vp \in H^1(\Om): \vp \geq 0,\ \vp = w \text{ on } \partial B_a}.
$$
\end{lemma}

\begin{proof}
Fix $\vp \in \mathcal{K}$. 
Pick $0 < \ve < a$. Find $\psi \in C^1(\Rn)$, $0 \leq \psi \leq 1$ such that $\psi =0$ on $B_{a-\ve}(0)$ and $\psi = 1$ on $B_a(0)^c$. Then $\psi \in W$. Define 
$$
\tilde \vp = \bcs
\vp & \text{on } B_a(0)^c,\\
w &\text{otherwise}.
\ecs
$$
Since $\vp|_{\partial B_a(0)} = w|_{\partial B_a(0)}$, $\tilde \vp \in H^1(\Rn \setminus B_{a-\ve})$. Define
$$
\phi = (\psi -1) w + \tilde \vp \in H^1(\Rn).
$$
Clearly $\phi \geq 0$ and $\phi = 0$ for $|x| \leq a - \ve$. Therefore $\phi \in V$. We have from problem \eqref{limitProblem}
$$
a_\Om(w, \vp - w) = a(w, \phi - \psi w) \geq \ang{-L, \phi - \psi w} = \ang{-L, \vp- w}_\Om.
$$
This holds for every $\vp \in \mathcal{K}$.
\end{proof}

The main tool is the following lemma:
\begin{lemma}[Comparison for the limit problem]
\label{comparisonLimitProblem}
Let $w_1$, $w_2$ be two solutions of the obstacle problem \eqref{limitProblem} for some $t > 0$ with $A_1$, $L_1$ resp. $A_2$, $L_2$. If $0 < A_1 < A_2$ and $L_1 \geq L_2 > 0$ then
$$
w_1 \leq w_2 \qquad \text{for all } x \neq 0.
$$
\end{lemma}

\begin{proof}
Let $\ve = \frac{A_2 - A_1}{3} > 0$. If $n \geq 3$, there exists $a > 0$ such that
$$
\abs{\frac{w_1(x)}{|x|^{2-n}} - A_1 t} < \ve \quad \text{and} \quad \abs{\frac{w_2(x)}{|x|^{2-n}} - A_2 t} < \ve \quad \text{for all } |x| \leq a.
$$
We can replace $\abs{x}^{2-n}$ by $-\half\log \abs{x}$ in dimension $n = 2$.
In particular, $w_1(x) \leq w_2(x)$ in $|x| \leq a$. Lemma \ref{limitProblemSatisfiesObstacleProblem} implies that $w_1$ and $w_2$ satisfy the obstacle problem on $\Omega = B_a(0)^c$,
$$
a(w_i, \vp - w_i)_\Omega \geq \ang{-L_i, \vp - w_i}_\Omega \qquad \forall \vp \in \mathcal{K}_i,\ i=1,2,
$$
where $\mathcal{K}_i = \set{\vp \in H^1(\Omega): \vp \geq 0,\ \vp = w_i \text{ on } |x| =a}$. Now we can use the comparison for the obstacle problem, see Corollary 5.2, chapter 4 in \cite{Rodrigues}, and that gives 
$$
w_1 \leq w_2 \qquad \text{in } \Omega.
$$
\end{proof}

Now we can finish proof of Theorem \ref{th:LPwellposed}.

Lemma \ref{comparisonLimitProblem} implies that for any $\ve > 0$ we can compare 
$$
U_{A - \ve, L}(x, t) \leq U(x, t) \leq U_{A+\ve, L}(x, t),,
$$
since $U_{A,L} (x, t)$ are solutions of the limit problem \eqref{limitProblem}.

After taking limit $\ve \to 0$ we conclude that
$$
U(x,t) = U_{A,L} (x, t).
$$
\end{proof}

\section{Uniform convergence of $u^\la$}

The following lemma was proven in \cite{KM}:
\begin{lemma}[cf. \cite{KM}*{Lemma 4.1}]
\label{th:average}
For given $g$ satisfying \eqref{boundong} and \eqref{statergodicg}, there exists a constant, denoted $\ang{\ov g}$, such that if $\Omega \subset \Rd$ is a bounded measurable set and if $\set{u^\ve}_{\ve >0}\subset L^2(\Omega)$ is a collection of functions such that $u^\ve \to u$ strongly in $L^2(\Omega)$ as $\ve \to 0$, then
\begin{align}
\label{eq:convToMean}
\lim_{\ve\to 0} \int_\Omega \ov{g(x/\ve, \om)} u^\ve(x) \dx = \int_\Omega \ang{\ov g} u(x) \dx \qquad \text{a.e. $\om$}
\end{align}
\end{lemma}

Recall that in section~2, we fixed $\om$ for which \eqref{eq:convToMean} holds.  The goal of this section is proving the following theorem:
\begin{theorem}
\label{th:uconvergence}
The functions $u^\la$ converge to $U_{A, L}$ as $\la \to \infty$ locally uniformly on $\Rd \setminus \set{0} \times [0, \infty)$, where $U_{A,L}$ is the unique solution of the limit problem from Theorem~\ref{th:LPwellposed}, with $A = C_*$ from Lemma~\ref{constantSingularity}, and $L = \ang{\ov{g}}$ from Lemma~\ref{th:average}.
\end{theorem}

\begin{proof}
We need to prove uniform convergence on sets $\set{\abs{x} \geq \ve,\ 0 \leq t \leq T}$. Hence fix $T > 0$ and $\ve > 0$.
The functions $u^\la(\cdot, t)$ satisfy the obstacle problem \eqref{TOP} for each $t \geq 0$ and $\la > 0$.

Using Lemma~\ref{th:boundaryBound}, we can find $\rho_2 > 0$ and $t_0 > 0$ such that 
\begin{align}
\label{largeTimes}
\Omega_t(u) \subset B_{\rho_2 \mathcal{R}(t)} \qquad \text{for all } t \geq t_0
\end{align}
and
\begin{align*}
\Omega_t(u) \subset B_{\rho_2} \qquad \text{for all } 0 \leq t \leq t_0
\end{align*}
Inclusion \eqref{largeTimes} is preserved under rescaling
 when $n \geq 3$.
In dimension $n = 2$, we have
\begin{align*}
\Omega_t(u^\la) \subset   B_{\rho_2 \frac{\mathcal{R}(\la t)}{\mathcal{R}(\la)}} \qquad \text{for all } t \geq t_0/ \la.
\end{align*}
But 
\begin{align*}
\frac{\mathcal{R}(\la t)}{\mathcal{R}(\la)} \leq \frac{\mathcal{R}(\la T)}{\mathcal{R}(\la)} \sim \pth{\frac{T \log \la}{\log \la T}}^{1/2} \to T^{1/2} \quad \text{as } \la \to \infty.
\end{align*}
Thus we can find $\hat \rho > 0$ such that $\Omega_t(u^\la) \subset \Omega := B_{\hat \rho} (0)$ for all $0 \leq t \leq T$ and $\la > 1$. Define $\Omega_\ve := \set{x \in \Omega : \abs{x} \geq \ve}$ and $Q_\ve:= \Omega_\ve \times [0,T]$.

Now find $\la_0 > 1$ such that $K^\la \subset B_{\ve/2}(0)$ for all $\la \geq \la_0$. Since $u^\la(\cdot, t)$ is a solution of the rescaled obstacle problem \eqref{TOP}, we can use the regularity estimates for the obstacle problem from \cite{Rodrigues}*{Theorem 5.2.4} to conclude that
\begin{align*}
\no{\Delta u^\la(\cdot, t)}_{L^q(\Omega_{\ve/2})} \leq \no{\ov{g^\la}}_{L^q(\Omega_{\ve/2})} \qquad \text{for all } 1 \leq q \leq \infty.
\end{align*}

Let $p_+$ be the radially symmetric supersolution from the proof of Lemma~\ref{th:boundaryBound}. Since the rescaled $p_+^\la$ converges uniformly on the set $Q_{\ve/2}$ to the function $V_{A,L}$ as $\la \to \infty$ for some $A$ and $L$ by Lemma~\ref{th:radialConvergence}, we can find a constant $C_1$ such that $p_+^\la \leq C_1$ on $Q_{\ve/2}$ for all $\la \geq \la_0$. Now recall that due to Theorem~\ref{th:equivalency} we can express $u^\la$ as
\begin{align}
\label{ulaint}
u^\la(x,t) = \int_0^t v^\la(x,s) \d s \leq \int_0^t p_+^\la(x,s) \d s \leq C_1 T \quad \text{for } (x,t) \in Q_{\ve/2}. 
\end{align}
In particular, $\no{u^\la(\cdot,t)}_{L^2(\Omega_{\ve/2})}$ is bounded uniformly in $t \in [0,T]$ and $\la \geq \la_0$.

Thus we can use the standard elliptic regularity results (see \cite{LU}, for instance) to find constants $0 < \al < 1$ and $C_2$, independent of $t \in [0,T]$ and $\la \geq \la_0$, such that
\begin{align*}
\begin{aligned}
\no{u^\la(\cdot, t)}_{H^2(\Omega_\ve)} &\leq C_2,\\
\no{u^\la(\cdot, t)}_{C^{0,\al}(\Omega_\ve)} &\leq C_2.
\end{aligned}
\qquad \text{for all } 0 \leq t \leq T,\ \la \geq \la_0.
\end{align*}

Using \eqref{ulaint} again, we have $\abs{u^\la(x,t) - u^\la(x,s)} \leq  C_3 \abs{t-s}$ and we conclude that 
\begin{align*}
\no{u^\la}_{C^{0,\al}(Q_\ve)} \leq C_4(C_2,C_3),
\end{align*}
for all $\la \geq \la_0$. 

Using the standard diagonalization argument and Arzel\`{a}-Ascoli, we can find a subsequence $u^{\la_k}$ that converges locally uniformly on sets $Q_\ve$, $\ve > 0$, to a function $\cl{u}$ as $k \to \infty$.
Due to the compact embedding of $H^1$ in $H^2$, the uniqueness of the limit and the bound on the $H^2$-norm implies that also $u^{\la_k}(\cdot,t) \to \cl{u} (\cdot, t)$ in $H^1$-norm on sets $Q_\ve$ as $\la \to \infty$ for every $0 \leq t \leq T$, $\ve > 0$.

Finally, in the two following lemmas, we will show that $\cl{u}$ is a solution of the limit problem \eqref{limitProblem}. Using the uniqueness of the limit problem, Theorem~\ref{th:LPwellposed}, we conclude that the convergence is not restricted to a subsequence and we have 
$$u^\la \to U_{C_*, \ang{\ov g}}, \qquad \la \to \infty,$$
locally uniformly on $\Rd \setminus \set{0} \times [0, \infty)$, which concludes the proof of Theorem~\ref{th:uconvergence}.

\begin{lemma}
$\cl w = \cl u(\cdot, t)$ satisfies
\begin{align*}
a(\cl w, \phi) &\geq (-L, \phi), &\forall \phi &\in V,\\
a(\cl w, \psi \cl u) &= (-L, \psi \cl w), &\forall \psi &\in W,
\end{align*}
for each $0 \leq t \leq T$, where $L = \ang{\ov{g}}$ defined in Lemma~\ref{th:average} and $V$, $W$ were defined in \eqref{eq:limitProblemVW}.
\end{lemma}

\begin{proof}
Fix $t \in [0,T]$ and $\phi \in V$. We will denote $w^k \equiv u^{\la_k}(x,t)$. 
Then there is $\ve > 0$ such that $\phi = 0$ in $B_\ve(0)$. There is also $k_0$ such that if $k \geq k_0$, we have $\Om_0^{\la_k} \subset B_\ve(0)$. Define sequence $\vp^k = \phi + w^k$. Clearly $\phi = 0$ on $\Gamma_0^{\la_k}$ if $k \geq k_0$ and thus $\vp^k \in \mathcal{K}_t^{\la_k}$. Since $u^{\la_k}$ satisfies \eqref{TOP}, we have
$$
a(w^k, \phi) = a(w^k, \vp^k - w^k) \geq \ang{-\ov{g^{\la_k}}, \vp^k - w^k} = \ang{-\ov{g^{\la_k}}, \phi}. 
$$
The mapping $w \mapsto a(w, \phi)$ is a linear functional on $H^1$ and hence we have $a(w^k, \phi) \to a(\cl w, \phi)$ as $k \to \infty$ because $w^k \to \cl w$ $H^1$-strongly. Lemma~\ref{th:average} also yields $\ang{-\ov{g^{\la_k}}, \phi} \to \ang{-L, \phi}$. Hence we conclude
$$
a(\cl w, \phi) \geq \ang{-L, \phi}.
$$

Now fix $\psi \in W$ such that $0 \leq \psi \leq 1$. There is $B_\ve(0)$ in which $\psi = 0$ and $k_0$ as above. Define $\vp_k = (1- \psi) w^k$. Again $\vp_k  \in \mathcal{K}_t^{\la_k}$ if $k \geq k_0$. Now
$$
a(w^k, \psi w^k) = -a(w^k, \vp^k -w^k) \leq -\ang{-\ov{g^{\la_k}}, \vp^k - w^k} = \ang{-\ov{g^{\la_k}}, \psi w^k}.
$$
The fact that $w^k \to \cl w$ $L^2$-strongly implies $\psi w^k \to \psi \cl w$ in $L^2$ and by Lemma~\ref{th:average} 
$$
\ang{-\ov{g^{\la_k}}, \psi w^k} \to \ang{-L, \psi \cl w}.
$$
Due to the lower semi-continuity of the map $w  \mapsto a(w, \psi w)$ in $H^1$, we also obtain
$$
a(\cl w, \psi \cl w) \leq \liminf_{k\to \infty} a(w^k, \psi w^k) \leq \limto{k} \ang{-\ov{g^{\la_k}}, \psi w^k} = \ang{-L, \psi \cl u}.
$$
Finally, since $\psi \cl w \in V$, we have
\begin{align*}
a(\cl w, \psi \cl w) \leq \ang{-L, \psi \cl w} \leq a(\cl w, \psi \cl w). \qquad \qedhere
\end{align*}
\end{proof}

\begin{lemma}
\label{singularityU}
We have
$$
\lim_{|x|\to 0} \frac{\cl u(x, t)}{U_{C_*,L}(x,t)} = 1,
$$
for every $t \geq 0$, where $C_*$ is the constant from Lemma \ref{constantSingularity}.
\end{lemma}

\begin{proof}
We will use comparison with the sub- and supersolutions from Lemma \ref{subSuperSolutions}.

Let $C_*$ be the constant from Lemma \ref{constantSingularity}. Fix $\ve > 0$. Then by Lemma \ref{constantSingularity} there exists $a > 0$ big enough such that 
$$
\abs{\frac{P(x)}{a^{2-n}} - C_*} < \ha{\ve} \qquad \text{for } |x| = a.
$$
The set $\set{|x| = a}$ is a compact subset of $\Rn \setminus K$ and by the near field limit, Theorem \ref{nearFieldLimit}, there is $t_0 > 0$ large enough so that
$$
\abs{\frac{v(x,t)}{a^{2-n}} - \frac{P(x)}{a^{2-n}}} < \ha{\ve} \qquad \text{for }|x| = a, t \geq t_0.
$$
Therefore we have
$$
\abs{\frac{v(x,t)}{a^{2-n}} - C_*} < \ve \qquad \text{for }|x| = a, t \geq t_0.
$$
Let $p_+$ and $p_-$ be the radially symmetric supersolution, resp. subsolution, satisfying the boundary conditions
\begin{align*}
\frac{p_\pm }{a^{2-n}} &= C_* \pm \ve \qquad\text{on } |x| = a,\\
b_+ &= \max_{x \in \Gamma_{t_0}(v)} |x|,
& b_- &= \min_{x \in \Gamma_{t_0}(v)} |x|,\\
L_+ &= \ov M, & L_-&= \ov m.
\end{align*}
Then by comparison 
$$
p_-(x, t - t_0) \leq v(x,t) \leq p_+(x, t - t_0) \qquad \text{for } |x| \geq a, t \geq t_0.
$$
Lemma~\ref{th:radialConvergence} yields
\begin{align*}
p_\pm^\la \to V_\pm := V_{C_* \pm \ve, L_\pm} \qquad \text{as } \la \to \infty,
\end{align*}
locally uniformly on $\set{\abs{x} > 0,\ t \geq 0}$.

Using the formula \eqref{rescaledU} for $u^\la$ we have for $\la^{1/n} |x| \geq |a|$ and $\la \geq t_0/t$ 
\begin{align*}
u^\la(x,t) &=  \int_0^{t} v^\la(x, s) \d s\\
&\leq \int_0^{t_0/\la} v^\la(x, s) \d s + \int_{t_0/\la}^t p^\la_+(x,s- t_0/\la) \d s\\
&=\int_0^{t_0/\la} v^\la(x, s) \d s + \int_0^{t-t_0/\la} p^\la_+(x,s) \d s
\end{align*}
The first term $\to 0$ as $\la \ri$ since $v^\la(x,s) \leq C|x|^{2-n}$ by Lemma~\ref{th:boundaryBound}. The locally uniform convergence of $p_+^\la$ then implies that the second integral converges to $\int_0^t V_+(x,s) \d s$  as $\la \ri$ by the bounded convergence theorem.
The same argument can be used to bound $u^\la$ from below using $p_-$ and we finally have
\begin{equation}
\label{eq:squeezeuinfty}
\begin{aligned}
\int_0^t V_-(x,s)\d s &\leq \liminf_{\la\ri} u^\la(x,t) \\
&\leq \cl u(x,t) \leq \limsup_{\la\ri} u^\la(x,t) \leq \int_0^t V_+(x,s) \d s.
\end{aligned}
\end{equation}
for all $(x,t) \in (\Rn \setminus \set{0})\times [0,\infty)$.

Now suppose that $n \geq 3$. The integrals on both sides can be found explicitly as in the proof of Theorem~\ref{th:LPwellposed} and 
$$
\int_0^t V_\pm(x,s) \d s = U_\pm(x,t) := U_{C_* \pm \ve, L_\pm}(x,t) ,
$$
with 
$$
\lim_{|x| \to 0} \frac{ U_{C_* \pm \ve, L_\pm}(x,t)}{|x|^{2-n}} = (C_* \pm \ve) t.
$$
Then we divide \eqref{eq:squeezeuinfty} by $|x|^{2-n}$ and take the limit $|x| \to 0$, which leads to 
$$
(C_* - \ve)t \leq \liminf_{|x|\to 0} \frac{\cl u(x)}{|x|^{2-n}}\leq \limsup_{|x|\to 0} \frac{\cl u(x)}{|x|^{2-n}} \leq  (C_* + \ve) t.
$$
We obtain an analogous result in dimension $n = 2$, after replacing $\abs{x}^{2-n}$ by $- \log \abs{x}$.

The proof is complete because $\ve > 0$ was arbitrary.
\end{proof}
This finishes the proof of Theorem~\ref{th:uconvergence}.
\end{proof}

\section{Convergence of $v$ and the free boundary}

In this final section, we want to show the locally uniform convergence of $v^\la$ and $\Gamma(v^\la)$ as $\la \to \infty$. Define the half-relaxed limits
\begin{align*}
v^*(x,t) &= \limsup_{(y,s),\la\to (x,t),\infty} v^\la(y,s),\\
v_*(x,t) &= \liminf_{(y,s),\la\to (x,t),\infty} v^\la(y,s),\\
\end{align*}
in $\set{|x| \neq 0,\ t \geq 0}$.
Let $V(x,t) = V_{C_*, L}(x,t)$ to be the radially symmetric solution with delta source at 0 from Lemma~\ref{th:radialConvergence},
where $C_*$ is the constant from Lemma \ref{constantSingularity} and $L = \ang{\ov{g}}$. Clearly
$$
\Om_t(V) = \set{x :0 < |x| < \rho(t)}.
$$
And by continuity, $V_* = V = V^*$.

Our goal is proving a theorem similar to \cite{KM}*{Theorem 4.4}.
\begin{theorem}
\label{uniformConvergenceOfGamma}
$\set{\Gamma(v^\la)}_\la$ converges to $\Gamma(V)$ locally uniformly with respect to the Hausdorff distance and $v^\la$ converges locally uniformly to $V$ in $\pth{\Rd \setminus \set{0}} \times [0, \infty)$ as $\la \to \infty$, and furthermore
$$
v_* = v^* = V.
$$
\end{theorem}

We will use the ideas developed in \cite{KM}. The proof is somewhat less technical because $V$, the limit of $v^\la$, is smooth. 

First, we will collect some necessary technical results in the spirit of \cite{KM}:
\begin{remark}
\label{rem:viscosityObservations}
\begin{enumerate}
\item Since $v$ is a viscosity supersolution, $v^\la(\cdot, t)$ is superharmonic in $\Omega_t(v^\la)$ and therefore also $v_*(\cdot, t)$ is superharmonic in $\Omega_t(v_*)$. 
\item $u(\cdot,t)$ is a weak solution of $-\Delta w \leq 0$ for every $t > 0$ and thus $v(x,t) \geq \ov t u (x,t)$.
\item The construction of the viscosity solution $v$ implies that $v^\la$ is subharmonic in $\pth{K^\la}^c$ (see \cite{KM}*{Proof of Theorem 3.1, step 2}). Therefore $v^*$ is subharmonic in $\Rd \setminus \set{0}$. 
\end{enumerate}
\end{remark}

\begin{lemma}[cf. \cite{KM}*{Lemma 4.6}]
\label{convergenceOnBoundaryU}
Suppose $(x_k, t_k) \in \set{u^{\la_k} = 0}$ and $(x_k, t_k, \la_k) \to (x_\infty, t_\infty, \infty)$ with $x_0 \neq 0$. Then the following holds:
\begin{enumerate}
\item $(x_0,t_0) \in \set{U = 0}$.
\item If $x_k \in \Gamma_{t_k} (u^{\la_k})$ then $x_\infty \in \Gamma_{t_\infty} (U)$.
\end{enumerate}
\end{lemma}
\begin{proof}
See Lemma 4.6 in \cite{KM}.
\end{proof}

We need to identify the singularity of $v_*$ and $v^*$ at $x = 0$.
\begin{lemma}
\label{singularityOfV}
$v^*$ and $v_*$ defined above have a singularity at 0 with
$$
\lim_{\abs{x}\to 0+} \frac{v_*(x, t)}{V(x,t)} = 1, \qquad \lim_{\abs{x}\to 0+} \frac{v^*(x, t)}{V(x,t)} = 1
$$
for each $t > 0$.
\end{lemma}
\begin{proof}
Fix $\ve > 0$. Using comparison with radially symmetric subsolution $p_-$ and supersolution $p_+$ as in Lemma~\ref{singularityU}, with
\begin{align*}
\frac{p_\pm(x,t)}{a^{2-n}} = C_* \pm \ve \qquad \text{for $\abs{x} = a$},
\end{align*} 
we deduce that there are $a> 0$ and $t_0 > 0$ such
\begin{align*}
p_-(x, t - t_0) \leq v(x,t) \leq p_+(x, t- t_0),\qquad \text{for all $t \geq t_0$, $\abs{x} \geq a$}  
\end{align*} 
Since 
$$
p^\la_\pm(x,t - t_0/\la) \to V_{C_* \pm \ve, L_\pm} (x,t) \qquad \text{as } \la \ri,
$$
locally uniformly on the set $(\Rn \setminus \set{0}) \times [0, \infty)$.
The uniform convergence yields
$$
V_{C_* - \ve, L_-} (x,t) \leq v_*(x,t) \leq v^*(x,t) \leq V_{C_* + \ve, L_+} (x,t).
$$
We have the conclusion because $\ve >0$ was arbitrary.
\end{proof}

\begin{lemma}[cf. \cite{KM}*{Lemma 4.12}] 
\label{boundFromBelowVLambda}
There exists a constant $C_1= C_1(n,M)$ such that if $(x_0, t_0) \in \Omega(v^\la)$ and $B_r(x_0) \cap \Omega_0^\la = \emptyset$, we have
$$
\sup_{B_r(x_0)} v^\la(\cdot, t_0) \geq \frac{C_1 r^2}{t_0}.
$$
\end{lemma}
\begin{proof}
Recall that $v^\la(\cdot, t_0) \geq u^\la(\cdot, t_0) /t_0$ (Remark~\ref{rem:viscosityObservations}) and use Lemma 3.3 in \cite{KM}.
\end{proof}

\begin{lemma}[cf. \cite{KM}*{Lemma 4.10}]
\label{domainVstarV}
The following inclusion holds:
$$
\Omega(V) \subset \Omega(v_*).
$$
Moreover,
$$
v_* \geq V.
$$
\end{lemma}
\begin{proof}
Recall that $v(x,t) \geq \ov t u(x,t)$ for all $t > 0$, $x \neq 0$ and that the inequality is preserved under rescaling \eqref{eq:rescaling} and \eqref{rescaledU}.
As $u^\la$ converges to $U$ uniformly on sets $\set{|x| > \ve} \times [0, T]$, we have
$$
v_*(x,t) \geq \ov t U(x,t)
$$
and we see that $\Omega(V) = \Omega(U) \subset \Omega(v_*)$.

Recall that $v_*(\cdot, t)$ is superharmonic in $\Omega_t(v_*)$ for each $t$ (Remark~\ref{rem:viscosityObservations}) and behaves at zero as $\sim V(\cdot, t)$ by Lemma~\ref{singularityOfV}. For fixed $\ve > 0$, comparison of $v_*(\cdot, t)$ and $V_{C_* - \ve, L} (\cdot, t)$ yields
$$
v_*(x,t) \geq V_{C_* - \ve, L}(x,t), \quad \text{for every } t > 0,
$$
because  $\Omega(V_{C_* - \ve, L}) \subset \Omega(V)$. We conclude by taking the limit $\ve \to 0$.
\end{proof}

\begin{lemma}[cf. \cite{KM}*{Lemma 4.13 (ii)}]
\label{boundaryVstarV}
The following inclusion of boundaries holds:
$$\Gamma(v^*) \subset \Gamma(V).$$
\end{lemma}
\begin{proof}
Proof as in \cite{KM}, with the use of Lemma~\ref{convergenceOnBoundaryU} above.
\end{proof}

\begin{proof}[Proof of Theorem \ref{uniformConvergenceOfGamma}]
Since $\Omega_t(v^*)$ is bounded for every time $t > 0$ by Lemma~\ref{th:boundaryBound}, and $\Omega(V)$ is a simply connected set, Lemma~\ref{boundaryVstarV} also implies that in fact $\Omega(v^*) \subset \cl{\Omega(V)}$. 
In particular, we see that
$$
\cl{\Omega(v^*)} \subset \cl{\Omega(V)} \subset \Omega(V_{C_* + \ve, L}) \quad \text{for all } \ve > 0.
$$
Recall that $v^*(\cdot, t)$ is subharmonic in $R^n \setminus \set{0}$ for every $t > 0$ (Remark~\ref{rem:viscosityObservations}). 
Since $\lim_{|x| \to 0} \frac{v^*(x, t)}{V(x,t)}  = 1$ for all $t \geq 0$, we see that $v^*(x,t) \leq V_{C_* + \ve, L}(x,t)$ for every $\ve > 0$. After sending $\ve \to 0+$, we recover
$$
V(x,t) \leq v_*(x,t) \leq v^*(x,t) \leq V(x,t)
$$
and
$$
\Gamma(v_*) = \Gamma(v^*) = \Gamma(V).
$$

To prove the local uniform convergence of boundaries with respect to the Hausdorff distance, fix $0 < t_1 < t_2$ and let us denote
$$
\Gamma^\la = \Gamma(v^\la) \cap \set{t_1 \leq t \leq t_2}, \qquad \Gamma^\infty = \Gamma(V) \cap \set{t_1 \leq t \leq t_2}.
$$
We define the $\de$-neighborhood of a set $A \subset \Rd \times \R$,
\begin{align*}
U_\de(A) := \set{(x,t) : \dist((x,t), A) < \de}.
\end{align*}
We need to show that for every $\de > 0$ there exists $\la_0 > 0$ such that 
\begin{equation}
\label{eq:inclusionLambdaInfty}
\Gamma^\la \subset U_\de(\Gamma^\infty) \qquad \forall \la \geq \la_0.
\end{equation}
and
\begin{equation}
\label{eq:inclusionInftyLambda}
\Gamma^\infty \subset U_\de(\Gamma^\la) \qquad \forall \la \geq \la_0.
\end{equation}

First, suppose that for some $\de > 0$, the inclusion \eqref{eq:inclusionLambdaInfty} fails infinitely often. In other words, suppose that there is $\de > 0$ and a sequence $\seq{\la}k$ with $\la_k \ri$ such that
$$
\Gamma^{\la_k} \cap \pth{U_\de(\Gamma^\infty)}^c \neq \emptyset.
$$
Accordingly, we choose a sequence of points $(x_k, t_k) \in \Gamma^{\la_k}$ whose distance from $\Gamma^\infty$ is bounded by $\de$ from below, 
\begin{equation}
\label{eq:pointDeltaDistance}
\dist(\Gamma^\infty,(x_k, t_k)) \geq \de.
\end{equation} 
Since the union $\bigcup_\la \Gamma^\la$ is bounded (recall Lemma~\ref{th:boundaryBound}), there is a converging subsequence $(x_{k_j}, t_{k_j}) \to (x_0, t_0)$ as $j \ri$. Moreover, $x_0 \neq 0$ by Corollary~\ref{th:boundaryBound} and also $t_1 \leq t_0 \leq t_2$. Hence Lemma~\ref{convergenceOnBoundaryU}(b) implies that $(x_0, t_0) \in \Gamma(U) = \Gamma(V)$ and therefore $(x_0,t_0) \in \Gamma^\infty$. But that is a  contradiction with \eqref{eq:pointDeltaDistance}.

To prove the second inclusion, \eqref{eq:inclusionInftyLambda}, we start by proving a pointwise result. Suppose that there is $\de > 0$, a point $(x_0, t_0) \in \Gamma^\infty$ and a sequence $\seq{\la}k$ with $\la_k \ri$ such that $\dist((x_0,t_0), \Gamma^{\la_k}) \geq \de/2$ for every $k$. That means that there is $r > 0$ such that the spacetime cylinder $D_r(x_0, t_0) := B_r(x_0) \times [t_0 - r, t_0 +r]$ is either
$$
D_r(x_0,t_0) \subset \set{v^{\la_k} > 0},
$$
or 
$$
D_r(x_0,t_0) \subset \set{v^{\la_k} = 0}.
$$
At least one of the inclusions holds for infinitely many $k$. Therefore we arrive at two possibilities:
\begin{enumerate}
\item
In the case of the first inclusion, Lemma \ref{boundFromBelowVLambda} and Harnack's inequality in $B_{r/2}(x_0)$ for all $t \in [t_0 - r, t_0 + r]$ yield
$$
\frac{C_1 r^2}{4t} \leq \sup_{B_{r/2}(x_0)} v^{\la_k}(\cdot, t) \leq C_2 \inf_{B_{r/2}(x_0)} v^{\la_k}(\cdot,t),
$$
where $C_2$ does not depend on $\la_k$ and $t$. We conclude that $v^* > 0$ in $B_{r/2}(x_0) \times [t_0 - r, t_0 + r]$ and that is a contradiction with $(x_0, t_0) \in \Gamma^\infty \subset \Gamma(v^*)$.
\item
In the second case, we clearly have that $v_* = 0$ in $D_r(x_0, t_0)$ but that is a contradiction with $(x_0, t_0) \in \Gamma^\infty \subset \Gamma(v_*)$.
\end{enumerate}
We conclude that for each $\de > 0$ and $(x_0,t_0) \in \Gamma^\infty$ there is $\la_0$, depending on $(x_0,t_0)$, such that $\dist((x_0, t_0), \Gamma^\la) < \de/2$ for all $\la \geq \la_0$. But if $(x,t) \in \Gamma^\infty$ such that $\abs{(x,t) - (x_0,t_0)} < \de/2$, we have 
$$
\dist((x, t), \Gamma^\la) \leq \dist((x_0, t_0), \Gamma^\la) + \abs{(x,t) - (x_0,t_0)} < \de \quad \forall \la \geq \la_0.
$$
The set $\Gamma^\infty$ is compact and therefore can be covered by finitely many open sets,
$$
\Gamma^\infty \subset \bigcup_{1\leq j \leq j_0} \set{(x,t) : \abs{(x,t) - (x^j_0,t^j_0)} < \de/2}.
$$
 We can set $\la_0 = \max_{1\leq j \leq j_0} \la^j_0$ and we conclude that
$$
\Gamma^\infty \subset U_\de(\Gamma^\la) \qquad \forall \la \geq \la_0.
$$

\end{proof}

\textit{Acknowledgments.} This work is part of my doctoral research. I would like to sincerely thank my thesis advisor, Inwon C. Kim, for suggesting the problem and for her invaluable advice and support.

This is a preprint of an article submitted for consideration in the Interfaces and Free Boundaries, which is available online at:

\verb+http://www.ems-ph.org/journals/journal.php?jrn=ifb+

\end{document}